\documentclass[12pt]{article}
\usepackage{graphicx}
\usepackage{color}
\usepackage{multirow,multicol}
\usepackage{epstopdf}
\usepackage{amsmath,amssymb,amsthm}

\newtheorem{theorem}{Theorem}[section]
\newtheorem{lemma}[theorem]{Lemma}
\newtheorem{corollary}[theorem]{Corollary}
\newtheorem{proposition}[theorem]{Proposition}
\newtheorem{definition}[theorem]{Definition}

\newtheorem{remark}[theorem]{Remark}

\numberwithin{equation}{section}

\newcommand{\diag}{\mathrm{diag}\,}

 \newcommand{\vp}{\varphi}
\newcommand{\D}{\displaystyle}

\newcommand{\ig}{\int_\Gamma}

\newcommand{\suml}{\sum\limits}

\newcommand{\liml}{\lim \limits}

\newcommand{\im}{\mathrm{im}\,}
\newcommand{\ve}{\varepsilon}

\newcommand{\ima}{\textrm{Im}\,}

\newcommand{\bR}{{\bf R}}

\newcommand{\cA}{\mathcal{A}}

\newcommand{\cC}{\mathcal{C}}

\newcommand{\cG}{\mathcal{G}}

\newcommand{\cJ}{\mathcal{J}}
\newcommand{\cK}{\mathcal{K}}
\newcommand{\cL}{\mathcal{L}}
\newcommand{\cM}{\mathcal{M}}
\newcommand{\cN}{\mathcal{N}}

\newcommand{\cT}{\mathcal{T}}

\newcommand{\sC}{{\mathbb C}}

\newcommand{\sM}{{\mathbb M}}
\newcommand{\sN}{{\mathbb N}}

\newcommand{\sR}{{\mathbb R}}

\newcommand{\sZ}{{\mathbb Z}}

\begin{document}

\begin{center}
{\Large\textbf{The stability of the Nystr\"om method for\\ double
layer potential equations}\footnote{This work is partially supported
by the Universiti Brunei Darussalam, Grant UBD/GSR/S\&T/19}}

\vspace{5mm}

\textbf{Victor D. Didenko\footnote{E-mail: diviol@gmail.com} and Anh
My Vu\footnote{E-mail: anhmy7284@gmail.com}}

\vspace{2mm}

Universiti Brunei Darussalam, Bandar Seri Begawan, BE1410  Brunei;
diviol@gmail.com

 \end{center}



\begin{abstract}
The stability of the Nystr{\"o}m method for the double layer
potential equation on simple closed piecewise smooth contours is
studied. Necessary and sufficient conditions of the stability of the
method are established. It is shown that the method under
consideration is stable if and only if certain operators associated
with the opening angles of the corner points are invertible.
Numerical experiments show that there are opening angles which cause
instability of the method.
\end{abstract}





\date{}



  \section{Introduction\label{s1}}
Boundary integral equations are widely used in approximate solution
of partial differential equations. For example, consider the
Dirichlet problem for the Laplace equation
  \begin{equation}\label{eqn1}
  \begin{aligned}
\Delta u(x,y)=&\,0,\quad (x,y)\in D\\
 u(x,y) =&f(x,y), \quad (x,y)\in \Gamma
\end{aligned}
 \end{equation}
where $D$ is a simply connected domain of $\sR^2$ and $\Gamma$ is
the boundary of $D$. It is well known that the solution of the
problem \eqref{eqn1} can be reduced to the solution of a double
layer potential equation.

On the contour $\Gamma$ consider the double layer potential equation
 \begin{equation}\label{eqn2}
(Ax)(t) = x(t) + \frac{1}{\pi} \ig x(\tau) \frac{d}{d n_{\tau }}\log
|t-\tau |d \, \Gamma_{\tau } + (Tx)(t) = f(t), \quad t=x+iy \in
\Gamma.
\end{equation}
where $n_\tau$ refers to the outer normal to $\Gamma$ at the point
$\tau\in \Gamma$ and $T$ is a compact operator in the corresponding
space. Note that concrete form of $T$ depends on the boundary value
problem considered. In particular, $T=0$ if one uses the following
representation
 $$
 u(z) :=\frac{1}{\pi} \ig x(\tau) \frac{d}{d n_{\tau }}\log |z-\tau
|\, d \Gamma_{\tau }, \quad z\in D,
 $$
in order to reduce the boundary value problem \eqref{eqn1} to the
integral equation \eqref{eqn2} \cite{Atkinson1997}. The operator
\eqref{eqn2} has been intensively studied and there is a vast
literature concerning the properties of these operators considered
on various contours and in various functional spaces \cite{Co:1988,
Kre:2014, RSS:2011, Ver:1984}. In particular, it is known that if
$\Gamma$ is a smooth curve, then the double layer potential operator
 \begin{equation}\label{eqn3}
 V_\Gamma x(t):= \frac{1}{\pi} \ig x(\tau) \frac{d}{d n_{\tau }}\log |t-\tau
|d \, \Gamma_{\tau }
 \end{equation}
is compact on $L^p$ spaces, and this fact essentially simplifies the
stability investigation for many approximation methods under
consideration. In fact, for smooth curves if the operator $A$ of
\eqref{eqn2} is invertible, then the corresponding approximation
method is stable provided that there is a "good" convergence of the
approximation operators to the operator $A$. On the other hand, if
$\Gamma$ possesses corner points, the operator \eqref{eqn3} is not
compact anymore. Therefore, now stability depends not only on the
invertibility of the operator $A$ of \eqref{eqn2} and on convergence
properties but also on additional parameters connected with the
method itself and with the opening angles of the corner points.
These features of approximation methods for equation \eqref{eqn2}
considered on piecewise smooth contours has been mentioned in a
number of works \cite{Atkinson1997, Chandler1987, DRS:1995,
Kre:1990}. In the present paper we consider the Nystr{\"o}m method
based on Gauss-Legendre quadrature formulas. Various modifications
of this method have been discussed in literature. It turns out that
these methods demonstrate a good convergence even in situation where
$\Gamma$ possesses a large sets of corner points \cite{Bre:2012a,
Bre:2010a, HeH:2013}. Nevertheless, a rigorous analysis of the
applicability of such methods is absent and one of the aims of this
work is to provide necessary and sufficient conditions of the
stability and to propose a method for their verification. It turns
out that stability depends on the invertibility of certain operators
which depend on the operator $A$, on the parameters of the method,
and on the opening angles of the corner points of $\Gamma$. In
particular, we found four angles in the interval $(0.1\pi,1.9\pi)$.
These angles are called critical and if the boundary $\Gamma$
possesses such corner points, the method is not stable. Note that
similar problems for the Sherman-Lauricella and Muskhelishvili
equations have been studied in \cite{DH:2011,DH:2011b, DH:2013a}.

Let us make a few technical remarks. Thus we identify any point
$(x,y)$ of $\sR^2$ with the corresponding point $z=x+iy$ in the
complex plane $\sC$. Let $S_\Gamma$ denote the Cauchy singular
integral operator on $\Gamma$,
  $$
  (S_\Gamma x)(t): =\frac{1}{\pi i}\ig  \frac{x(\tau)\,d\tau}{\tau -t}.
  $$
and let $M$ be the operator of complex conjugation,
$M\vp(t):=\overline{\vp(t)}$.

It is known \cite{Muskhelishvili1968} that the double layer
potential operator $V_\Gamma$ can be represented in the form
 $$
 V_\Gamma = \frac12(S_\Gamma + MS_\Gamma M).
 $$
Therefore, equation \eqref{eqn2} can be rewritten in the form
 \begin{equation}\label{eqn4}
 Ax=\left (I +\frac{1}{2} S_\Gamma + \frac{1}{2} MS_\Gamma M +T
\right )x=f.
 \end{equation}
Note that in this paper, equation \eqref{eqn2} is considered in the
space $L_p:=L^p(\Gamma, w)$ of all Lebesgue measurable functions $f$
satisfying the condition
 $$
||f||_{L^p}:= \left ( \int_\Gamma |f(t)|^p w(t)\,|dt| \right
)^{1/p}<\infty,\quad 1<p<\infty,
 $$
where $w(t):=\prod_{j=0}^{q-1}|t-\tau_j|^{\alpha_j}$,
$0<\alpha+1/p<1$ and $\tau_j,j=0,1,\ldots, q-1$ are the corner
points of $\Gamma$.

This paper is organized as follows. Section \ref{s2} is devoted to
Fredholm properties of the operator $A$. The results of this section
are well known and are presented here in order to make the paper
self contained and to introduce operators and functions which arise
in the study of approximation operators.

Sections \ref{s3} and \ref{s4} deals with the Nystr\"om method for
the equation \eqref{eqn2}. More precisely, we study the stability of
the Nystr\"om method based on composite Gauss-Legendre quadrature
rule
\begin{equation}\label{eqn5}
\int_0^1 u(s)ds \approx \suml_{l=0}^{n-1}\suml_{p=0}^{d-1}w_p
u(s_{lp})/n,
\end{equation}
where
\begin{equation*}
s_{lp}=\frac{l+\ve_p}{n},\quad l=0,1,\ldots , n-1,\; p=0,1,\ldots ,
d-1,
\end{equation*}
and $w_p$ and $0<\ve_0 < \ve_1 < \ldots < \ve_{d-1}$ are weights and
Gauss-Legendre points on the interval $[0,1]$. It's shown that the
method under consideration is stable if and only if certain
operators $B_{\omega_j,\delta , \ve}$ acting in the spaces of
sequences of complex numbers are invertible. It can be shown that
the operators arising belong to an algebra of Toeplitz operators
with matrix symbols. However, at present there are no efficient
criteria to verify whether operators from that algebra are
invertible or not. Therefore, we propose a numerical approach which
allows us to detect critical angles of the Nystr\"om method. Our
computations are restricted to the opening angles from the
$[0.1\pi,1.9\pi]$. The remaining opening angles can be also
considered but working with small angles and angles close to $2\pi$
requires much more effort and computational cost is high.

\section{Fredholm properties of the double layer potential equations\label{s2}}

It is well-known that the invertibility of the operator $A$ is a
necessary condition for the applicability of many numerical methods.
It depends on the curve $\Gamma$, on the compact operator $T$ and on
the space where the operator $A$ acts. In this section we present
certain conditions of the Fredholmness of the operator $A$
considered in the space $L^p(\Gamma,w)$, $1<p<\infty$ in the case of
piecewise smooth curves $\Gamma$. More precisely, let $\Gamma$ be a
simple piecewise smooth positively oriented contour in the complex
plane and let $\gamma: \sR \mapsto \sC$ be a $1$-periodic
parametrization of $\Gamma$. By $\cM_\Gamma$ we denote the set of
all corner points $\tau_0, \tau_1, \ldots , \tau_{q-1}$ of $\Gamma$
and assume that
 $$
\tau_j = \gamma (j/q),\; j=0,1,\ldots , q-1,
 $$
the function $\gamma$ is two times continuously differentiable on
each subinterval \\
$ (j/q,(j+1)/q)$ and
    $$
\left | \gamma ' (j/q+0 )\right | = \left | \gamma ' (j/q-0 )\right
|,\; j=0,1, \ldots , q-1.
   $$
Moreover, let $\omega_j$ denote angle between the two semi-tangents
at the corner point $\tau_j$, and let $\beta_j$ be the angle between
the right semi-tangent and the real axis $\sR$. Consider also the
contour
  $$
\Gamma_j:=\sR^- e^{i(\beta_j + \omega_j)} \bigcup \sR ^+
e^{i\beta_j},
 $$
where $\sR^-$ and $\sR^+$ are the positive semi-axes directed to and
away from the origin, respectively.

With each corner point $\tau_j$ we associate an operator $A_{\Gamma
_j}:L^2(\Gamma_j) \mapsto L^2(\Gamma_j)$ defined by $ A_{\Gamma_j}:
= I +V_{\Gamma_j}$ where $V_{\Gamma_j}$ is the double layer
potential operator on $\Gamma_j$.

Application of Theorem 1.9.5 of \cite{DS:2008} to the operator $A$
of \eqref{eqn4} leads to the following result
  \begin{proposition}\label{521}
Let $\Gamma$ be a simple closed piecewise smooth curve in the
complex plane $\sC$. Then the operator $A$ is Fredholm if and only
if all the operators $A_{\Gamma_j}$ are invertible for all
$j=0,1,\ldots, q-1$.
  \end{proposition}

Let $\sM$ and $\sM ^{-1}$ denote respectively the direct and the
inverse Mellin transforms, i.e.
\begin{align*}
(\sM f)(z) &=\int_0^{+\infty}x^{1/p+\alpha-zi-1}f(x)dx,\\
 (\sM ^{-1}
f)(x)&=\frac{1}{2\pi}\int_{-\infty}^{+\infty}x^{zi-1/p-\alpha}f(z)dz.
\end{align*}
 The Mellin convolution operator $\cM(b)$ with the symbol $b$ is
defined by
  \begin{equation*}
\cM(b)x(\sigma)=((\sM ^{-1}b\sM)x)(\sigma),
  \end{equation*}
and for some classes of symbols $b$, this operator can be
represented in the integral form
  \begin{equation}\label{eqn6}
\cM(b)x(\sigma)  = \int_0^{+\infty} {\bf k}\left (
\frac{\sigma}{s}\right )x(s) \frac{ds}{s},
 \end{equation}
where ${\bf k} =\sM^{-1}b$.

Consider now the operator $\cN_\omega : L^p(\sR ^+, t^{\alpha_j})
\to L^p(\sR ^+, t^{\alpha_j})$ defined by
  \begin{equation*}
(\cN_\omega (\phi))(\sigma) = \frac{1}{\pi i}\int_0^{+\infty}
\frac{\phi (s) ds}{s-\sigma e^{i\omega}}
  \end{equation*}
It is easily seen that $A_{\Gamma_j}$ is isometrically isomorphic to
the matrix operator $A_{\omega_j}:L^p(\sR^+, t^{\alpha_j})^2 \mapsto
L^p(\sR^+, t^{\alpha_j})^2$,
    \begin{equation}\label{eqn7}
A_{\omega_j} = \left (
\begin{array}{cc}
  I & (1/2)\big ( \cN_\omega - \cN_{2\pi-\omega}\big ) \\[1ex]
  (1/2)\big ( \cN_\omega - \cN_{2\pi-\omega}\big ) & I
  \end{array}
  \right ),
    \end{equation}
where
 $$
L^p(\sR^+, t^{\alpha_j})^2 := L^p(\sR^+, t^{\alpha_j})\times
L^p(\sR^+, t^{\alpha_j}),
 $$
and the corresponding isomorphism is given by the relation $A
\mapsto \eta A \eta ^-1$ with the mapping $\eta :L_p(\Gamma
_j,t^{\alpha_j}) \mapsto L^p(\sR^+, t^{\alpha_j})^2 $ defined by

$$
\eta (f)(s) = (f(se^{i(\beta_j +\omega_j)}),
f(se^{i\beta_j}))^T,\;s\in \sR ^+.
$$

It is well-known \cite{RSS:2011,DS:2008} that $\cN_\omega$ is the
Mellin convolution operator $\cM({\bf n}_\omega)$ with the symbol
   \begin{equation}\label{eqn8}
  {\bf n}_{\omega_j}(y)=\frac{e^{(\pi-\omega_j)y}}{\sinh \pi y},\quad
  y=z+\left (\frac{1}{p}+\alpha_j\right )i,\; z\in \sR.
  \end{equation}
This immediately leads to the formula
    \begin{equation*}
{\mathrm {smb}} \left ((1/2) (\cN_{\omega_j} -
\cN_{2\pi-\omega_j})\right)=\frac{\sinh (\pi-\omega_j)y}{\sinh \pi
y}
  \end{equation*}
  where $y$ as above.
  Thus
  \begin{equation}\label{smb}
 \mathrm {smb}\, A_{\omega_j}(y)=  \left (
\begin{array}{cc}
  \D 1 & \D\frac{\sinh (\pi -\omega_j)y}{\sinh \pi y} \\
  \D\frac{\sinh (\pi - \omega_j)y}{\sinh \pi y} & 1
  \end{array}
  \right   ).
 \end{equation}
Note  that the Mellin operator $(1/2)(\cN_{\omega_j} -
\cN_{2\pi-\omega_j})$ can be also represented in the integral form
  \eqref{eqn6} with the kernel $\mathbf{k}=\mathbf{k}_{\omega_j}$
  having the form
  \begin{equation}\label{eqn9}
  {\bf k}(z) = {\bf k}_{\omega_j} (z) = \frac{1}{\pi i}\frac{iz \sin
  \omega_j}{(1-ze^{i\omega_j})(1-ze^{-i\omega_j})}
  \end{equation}
  \begin{corollary} \label{cor2}
Let $\Gamma$ be a simple closed piecewise smooth contour satisfying
the conditions of Section \ref{s2}. Then the operator $A$ of
\eqref{eqn4} is Fredholm in the space $L^2(\Gamma)$.
\end{corollary}
  \begin{proof}
The matrix Mellin operator $A_{\omega_j}$ is invertible in
$L^2(\Gamma)$ if and only if its symbol \eqref{smb} is invertible.
The determinant of $\mathrm{smb}\,A_{\omega_j}$ is $1-\sinh ^2
(\pi-\omega_j)y/\sinh ^2 \pi y$, and it vanishes if and only if
$\sinh (\pi -\omega_j)y = \sinh \pi y$ or $\sinh (\pi -\omega_j)y =
-\sinh \pi y$. Consider, for example, the first of these equations
in the case $p=2$ and $\alpha_j=0$. Separating the real and
imaginary parts, one obtains the following system of equations
  \begin{align*}
    \cosh ((\pi - \omega_j)z) \sin \frac{\pi-\omega_j}{2}&=   \cosh (\pi z)\\
   \sinh ((\pi - \omega_j)z)\cos \frac{\pi-\omega_j}{2}&=0,
   \end{align*}
where $z\in\sR$. Since $\cos ((\pi-\omega_j)/2) \not = 0$ for any
$\omega_j \in (0,2\pi)$, the second equation of the system is
satisfied if $z=0 $ or $\pi-\omega_j=0$. If $z=0$, the first
equation of the system becomes $\sin ((\pi-\omega_j)/2)=1$ which has
no solution for $\omega_j \in (0,2\pi)$. On the other hand, if
$\pi-\omega_j=0$, the first equation becomes $\cosh (\pi z)=0$ which
obviously has no solution. Thus, the symbol of $A_{\omega_j}$ does
not vanish on the line $\sR + i/2$. Therefore, the operator
$A_{\omega_j}$ is invertible for any $j=0,1,\ldots , q-1$ and so are
the operators $A_{\Gamma_j}$. Now one can apply Proposition
\ref{521} and obtain Fredholmness of the operator $A$.
  \end{proof}

    \section{Stability of the Nystr\"om method\label{s3}}

From now on we consider our operators as acting on the space $L^2$
without weight, i.e. we set $L^2:=L^2(\Gamma,0)$. Therefore,
according to Corollary \ref{cor2}, the operator $A$ of \eqref{eqn4}
is Fredholm. Choose an $n\in\sN$ and assume that $d$ is a
non-negative integer such that $n>d+1$.  By $S_n^d(\Gamma)$ we
denote the space of the smoothest splines of degree $d$ on $\Gamma$
associated with the parametrization $\gamma:\sR\to \Gamma$, cf.
\cite{DH:2011}. Consider the two sets of points on $\Gamma$
\begin{equation*}
\tau_{lp}=\gamma \left (\frac{l+\ve_p}{n} \right ), \; t_{lp}=\gamma
\left (\frac{l+\delta_p}{n}\right ),\; l=0,1,\ldots,n -1;\,
p=0,1,\ldots,d-1.
\end{equation*}
where $0<\ve_0 < \ve_1<\ldots < \ve_{d-1}<1$ and $0<\delta_0
<\delta_1< \ldots < \delta_{d-1}<1$ are real numbers.

 If the integral operator $K$,
 \begin{equation*}
K \vp (t):= \int_\Gamma k(t,\tau) \vp(\tau)\,d\tau
\end{equation*}
has a sufficiently smooth kernel $k$ and if $\vp$ is a Riemann
integrable function, then we can approximate it by the quadrature
rule \eqref{eqn5}. Thus
      \begin{equation}\label{eqn10}
      \begin{aligned}
&\int_\Gamma k(t,\tau) \vp(\tau)\,d\tau =\int_0^1
    k(\gamma(\sigma),\gamma(s)) \vp(\gamma(s))\gamma'(s)\,ds  \\
 &
 \quad \approx     K^{(\ve,n)} \vp(t) =\sum_{l=0}^{n-1} \sum_{p=0}^{d-1} w_p
 k(t, \tau_{lp}) \vp(\tau_{lp}) \tau'_{lp}/n,
\end{aligned}
 \end{equation}
where $\tau'_{lp}=\gamma'((l+\ve_p)/n )$. In particular,
straightforward calculations show that for the kernel $k$ of the
double layer potential operator $V_\Gamma$, the limit
 $$
\lim \limits_{t\to \tau}k(t,\tau) = \frac{i\ima[\overline {\gamma
'(s)}\gamma'' (s)]}{\gamma '(s)|\gamma '(s)|^2},\; \tau=\gamma (s)
 $$
is finite for any $\tau \notin \cM_\Gamma$. Thus the kernel
$k=k(\tau,t)$ of the double layer potential operator $V_\Gamma$
behaves well and formula \eqref{eqn10} can be used in order to
approximate the operator $V_\Gamma$ even in the case where $\ve_p =
\delta_p$.

Let $Q^\delta_n: L_{\infty}(\Gamma) \mapsto S_n^d $ denote the
interpolation projection on the space $S_n^d$ such that
\begin{equation*}
Q_n^\delta x(t_{lp})=x(t_{lp}),\quad l=0,1,\ldots , n-1,\;
p=0,1,\ldots , d-1.
\end{equation*}
for all $x$ from the set $\bR (\Gamma)$ of all Riemann integrable
functions on $\Gamma$. Note that if none of $\delta_p$ is equal to
$0.5$, such projection operators $Q_n^delta$ exist and the sequence
$(Q_n^\delta)_{n\in \sN}:\bR (\Gamma) \mapsto L^2(\Gamma)$ converges
strongly to the corresponding embedding operator
\cite{Prossdorf1991}, viz.
\begin{equation}\label{eqn11}
\lim \limits_{n\to \infty} \|Q_n^\delta - f\|_{L^2(\Gamma)}=0,\quad
f \in \bR (\Gamma).
\end{equation}
Let $P_n: L^2(\Gamma) \mapsto S_n^d$ be the orthogonal projection
onto the spline space $S_n^d$. Recall that on the space $L^2$ the
sequence $(P_n)$ converges strongly to the identity operator.

Consider the Nystr\"om method for the double layer potential
equation \eqref{eqn2}. For simplicity, we drop the compact operator
$T$ and consider the equation
\begin{equation}\label{eqn12}
A_\Gamma x = (I+V_\Gamma)x = f.
\end{equation}
This simplifies the notation but does not influence the proofs of
main results on the stability of the corresponding method. An
approximate solution $x_n$ of \eqref{eqn12} can be derived from the
equations
\begin{equation}
Q_n^\delta A_\Gamma^{(\ve , n)}P_n x_n := Q_n^\delta P_n x_n +
Q_n^\delta V_\Gamma ^{(\ve , n)}P_n x_n = Q_n^\delta
 f, \quad x_n \in S^d_n, n\in \sN.
\end{equation}
These operator equations are equivalent to the following systems of
linear algebraic equations
\begin{equation}\label{eqn13}
\begin{aligned}
x(t_{kr}) & + \frac{1}{2\pi i}\suml_{l=0}^{n-1}\suml_{p=0}^{d-1} w_p
x(t_{lp})\left ( \frac{\tau '_{lp}}{\tau_{lp} -
t_{kr}}-\frac{\overline{\tau '_{lp}}}{\overline{\tau_{lp}}
-\overline
{t_{kr}}} \right )\frac{1}{n} \\
&= f(t_{kr}), \quad k=0,1,\ldots , n-1,\; r=0,1,\ldots , d-1.
\end{aligned}
\end{equation}
Let us consider examples of approximate solution of the equation
\eqref{eqn12} given on two different contours $\cL_1=\cL_1(\omega)$
and $\cL_2=\cL_2(\omega)$, $\omega\in(0,2\pi)$ in the case of
continuous and discontinuous right hand sides.
\begin{figure}[!tb]
   \centering
\includegraphics[height=45mm,width=60mm]{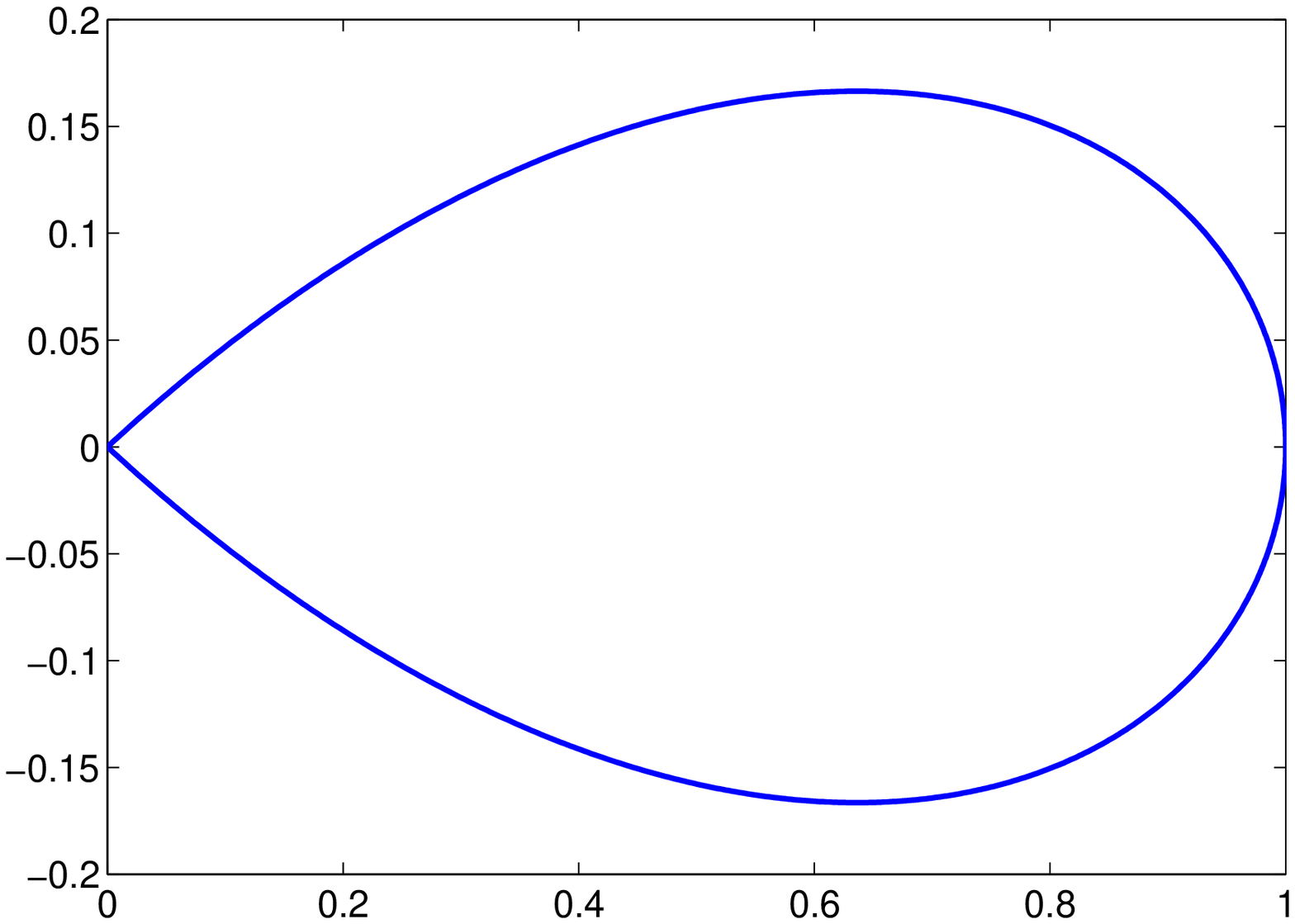}~%
 \includegraphics[height=45mm,width=60mm]{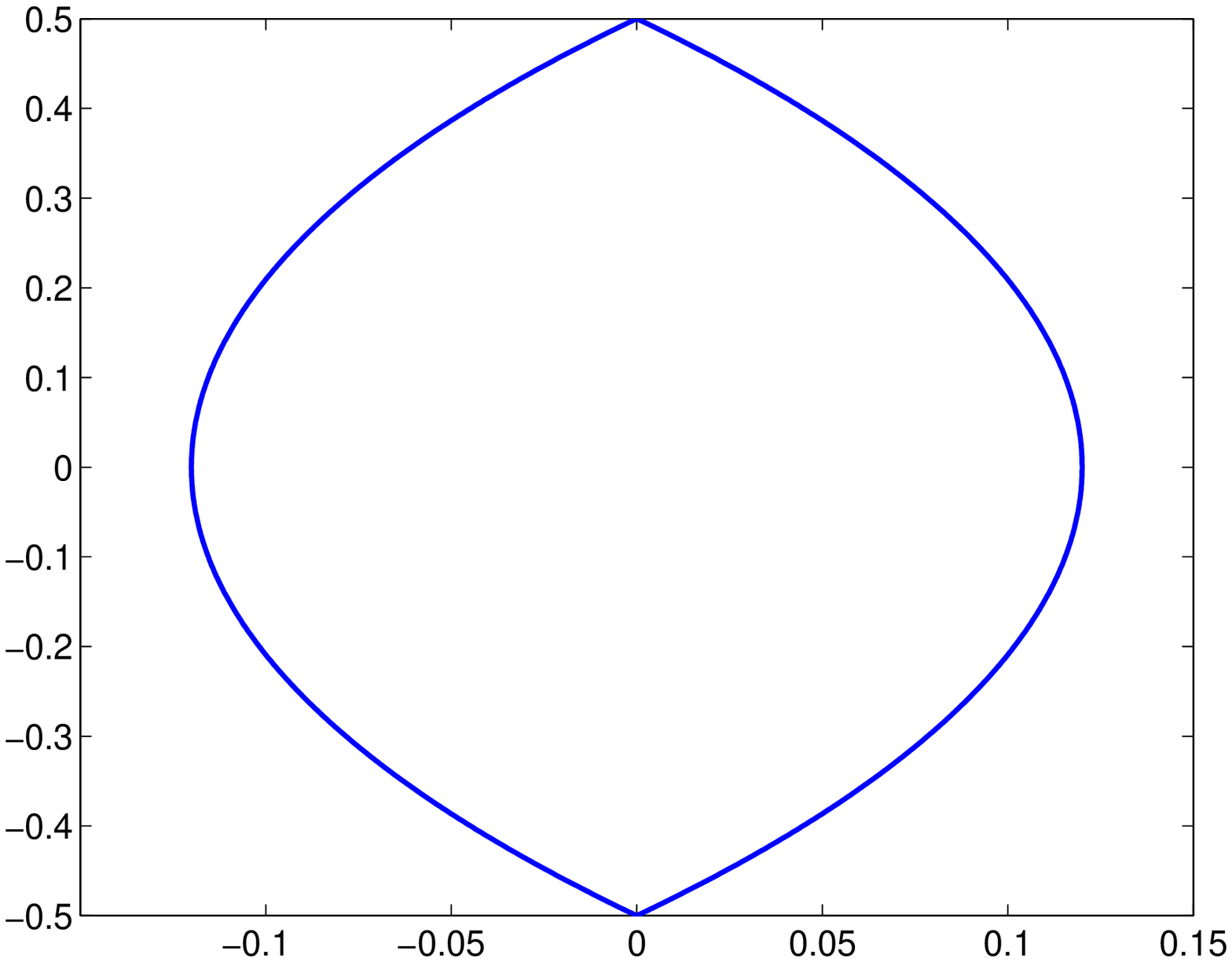}
 \caption{The curves $\cL_1$, left, and $\cL_2$, right. All corner points have
the same opening angle $\omega=0.3\pi$.}
 \label{Gamma12}
\end{figure}%
Let $f_1$ and $f_2$ be the following functions
 $$
f_1(z) = -z|z|,
 $$
and
 $$
f_2(z) = \begin{cases} -1+iz &\text{ if } \ima z <0 \\
\quad \! 1+iz &\text{ if } \ima z \geq 0 \end{cases}
$$
The curves $\cL_1$ and $\cL_2$ have, respectively, one and two
corner points of the magnitude $\omega\in (0,2\pi)$ each, and are
defined by
\begin{equation*}
 \cL_j:=\{ t\in\sC: t=\gamma_j(s), \quad s \in [0,1] \}, \quad
 j=1,2.
 \end{equation*}
where
\begin{align*}
 \gamma_1(s)&= \sin (\pi s) \exp (i \omega
(s-0.5)), \quad s \in [0,1],\\[1.5ex]
 \gamma_2(s)& = \begin{cases}
\D -\frac{1}{2} \cot  (\omega/2 )+ \frac{1}{2\sin (\omega/2)} \exp
(i\omega (2s-0.5)) & \text{ if} \; 0 \leq s \leq 1/2;
\\[1.5ex]
\D \frac{1}{2} \cot (\omega/2 )-\frac{1}{2\sin (\omega/2 )}\exp
(i\omega (2s-1.5)) & \text{ if}\;  1/2 < s \leq 1.
\end{cases}
\end{align*}
The graphs of the curves $\cL_1$ and $\cL_2$ which are used in our
examples below, are presented in Figure \ref{Gamma12}.

The right-hand side $f_1(z)$ is continuous on both curves $\cL_1$
and $\cL_2$, whereas $f_2(z)$ is discontinuous on both curves.
Moreover, one of the discontinuity points of $f_2$ coincides with
the angular point of $\cL_1$.
\begin{figure}[!th]
  \centering
\includegraphics[height=45mm,width=60mm]{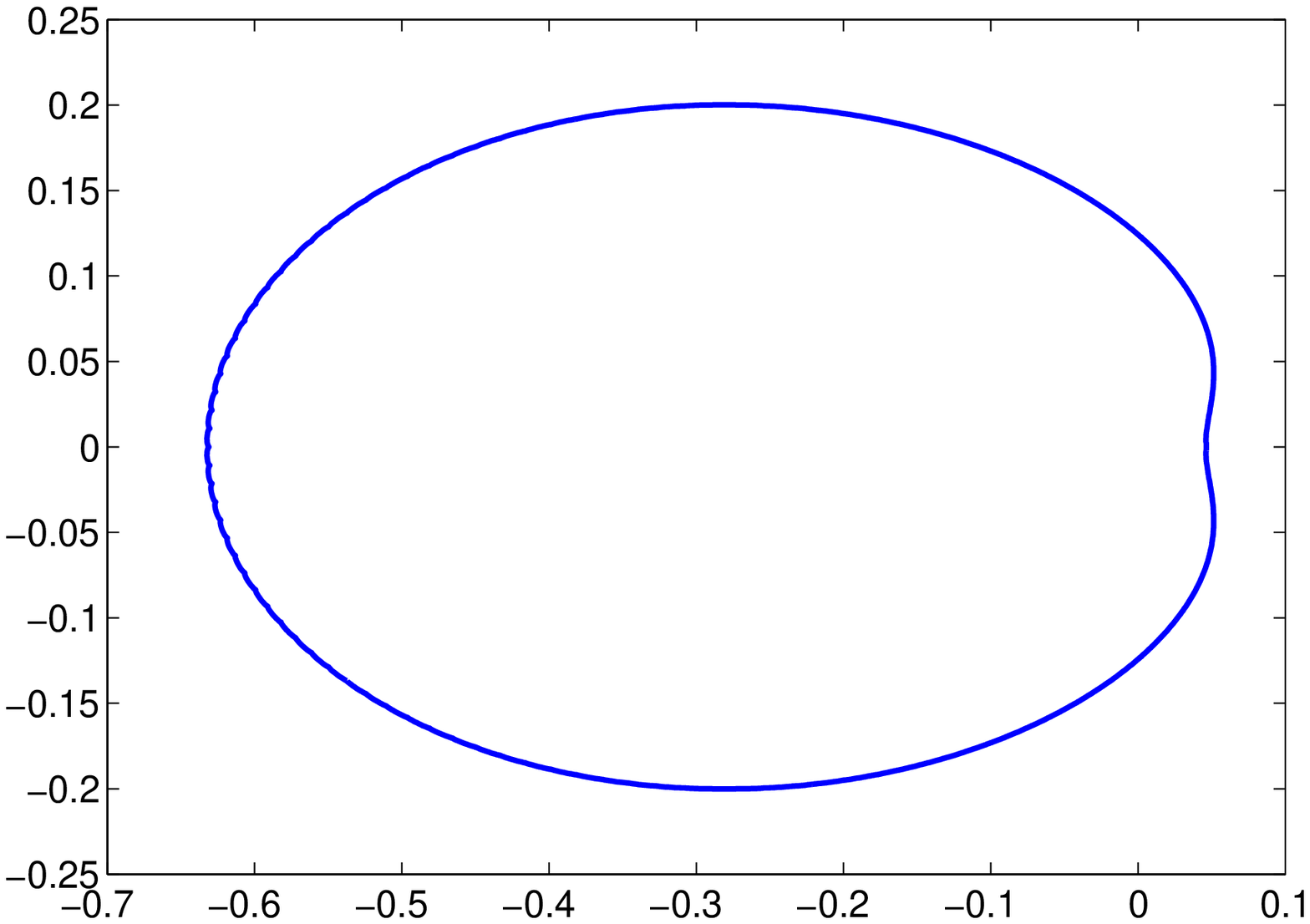}~%
 \includegraphics[height=45mm,width=60mm]{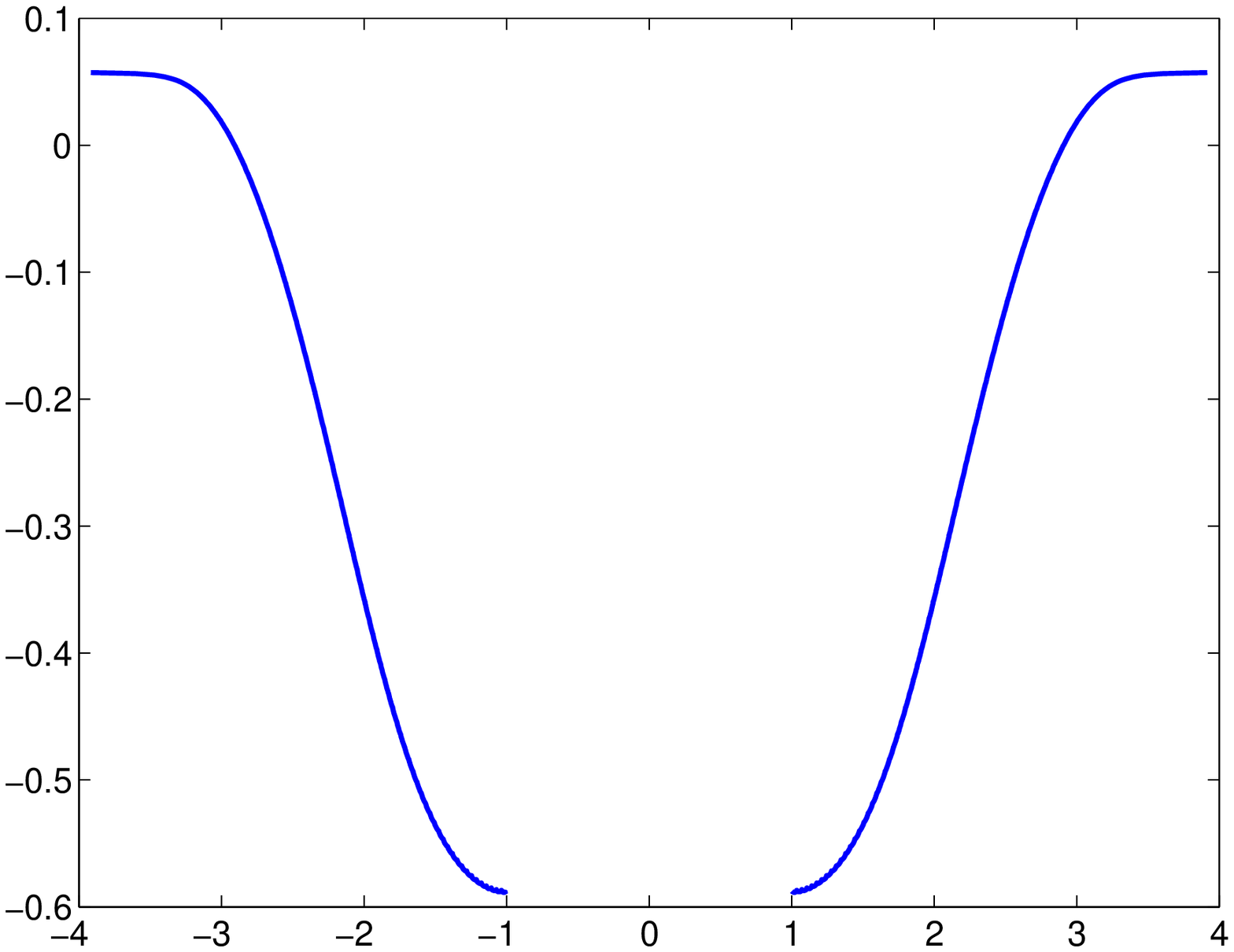}
\includegraphics[height=45mm,width=60mm]{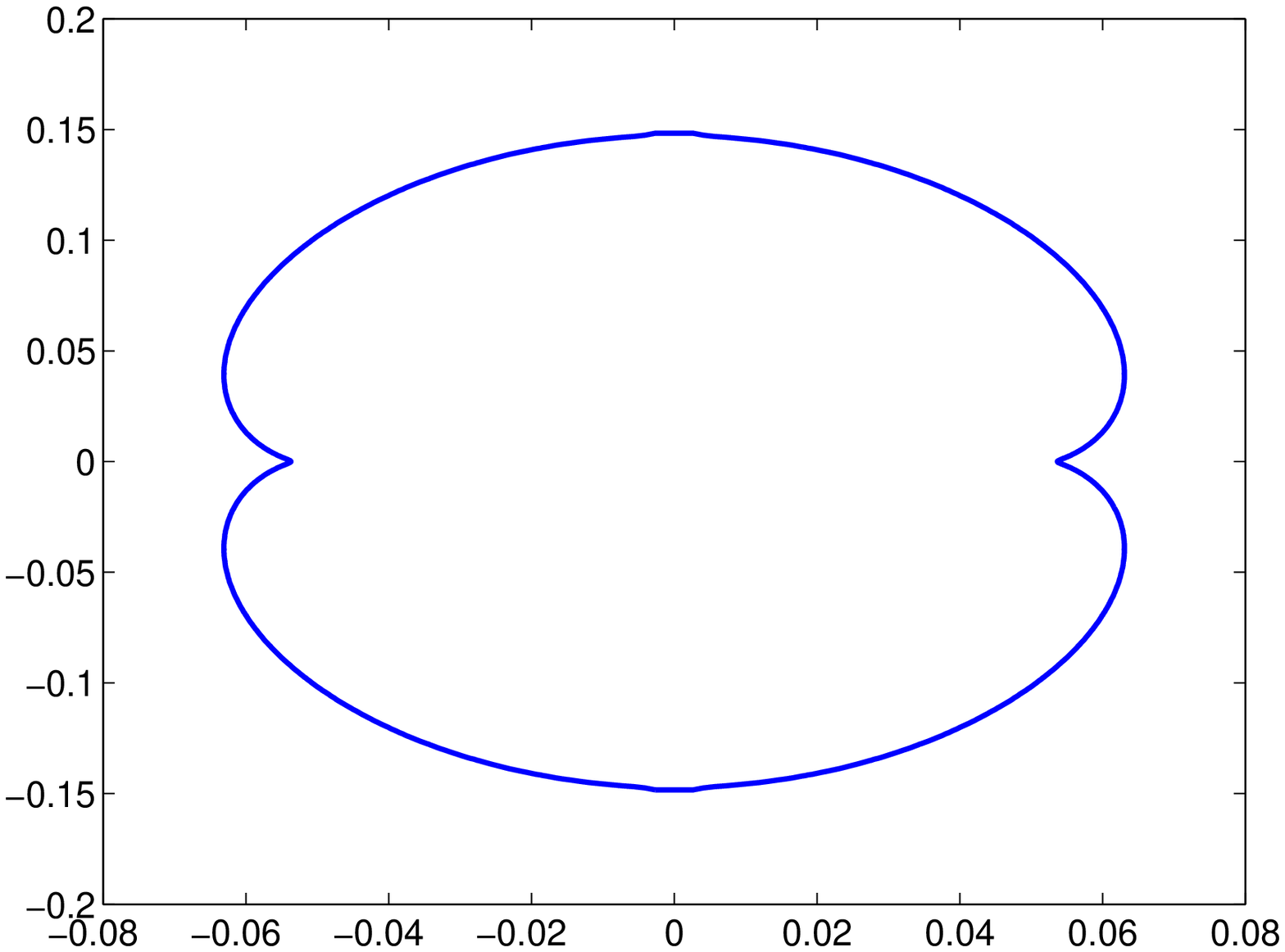}~%
 \includegraphics[height=45mm,width=60mm]{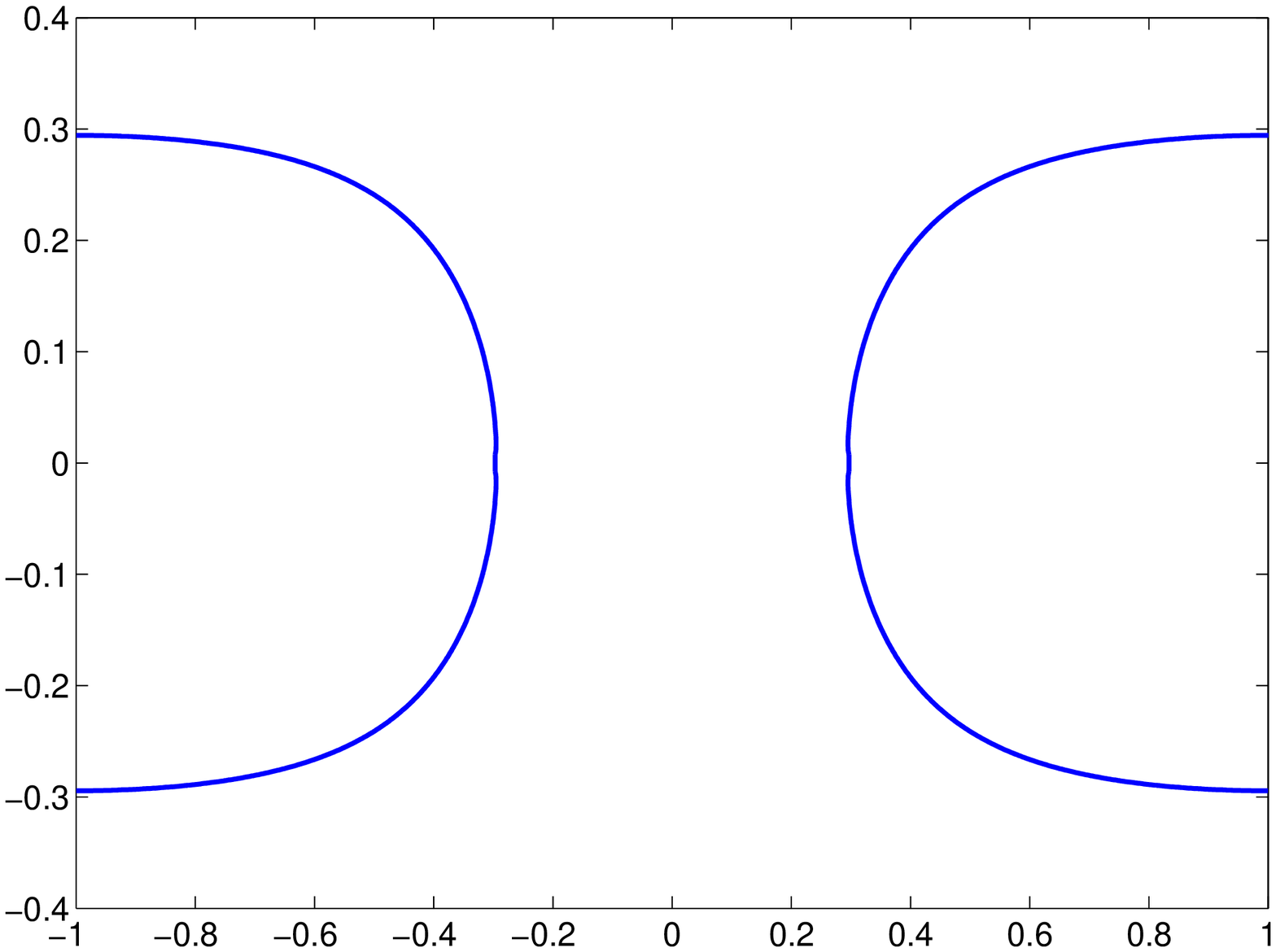}
   \caption{\sf Approximate solutions of \eqref{eqn12} with two different right hand sides
   and contours obtained by using method \eqref{eqn13} with $n=512, d=16$.
   Left: Solutions in the case of continuous r.-h.s. $f_1(z)$. Right:
   Solutions in the case of discontinuous r.-h.s. $f_2(z)$.
   First row: Equations on $\cL_1$. Second row: Equations on $\cL_2$.}
   \label{fig:app_sol}
\end{figure}
Approximate solutions of equation are obtained by the Nystr\"om
method \eqref{eqn13} with $d=16$ and $\ve_p=\delta_p$ and their
graphs are presented in  Figure \ref{fig:app_sol}. In the following
table, the term $E_n^{(f_k,\cL_j)}$ shows the relative error
$\|x_{2n}-x_n\|_2 / \|x_{2n}\|_2$ where $x_n$ is the approximate
solution of the equation \eqref{eqn12} for the contour
$\cL_j(0.3\pi),\; j=1,2$ with the right hand side $f_k$, $k=1,2$.
 \begin{center}
\begin{tabular}{|c|c|c||c|c|}
\hline
n & $E_n^{(f_1,\cL_1)}$ & $E_n^{(f_1,\cL_2)}$ & $E_n^{(f_2,\cL_1)}$ & $E_n^{(f_2,\cL_2)}$ \\
\hline $32$ & $2.5\times 10^{-3}$ & $2.6\times 10^{-3} $& $1.5\times 10^{-2}$&$ 2.0\times 10^{-2}$ \\
\hline $96$ & $8.3\times 10^{-4} $&$ 1.1\times 10^{-3}$ &$ 7.5\times 10^{-3}$&$1.3\times 10^{-2} $\\
\hline $256$ &$3.1\times 10^{-4} $&$ 2.1\times 10^{-4} $&$ 4.0\times 10^{-3}$&$7.3\times 10^{-3} $\\
\hline
\end{tabular}
\end{center}

\vspace{2mm}

\noindent It is worth noting that a better convergence rate can be
achieved by using certain modifications of the Nystr\"om method
\cite{Bre:2012a, Bre:2010a, HeH:2013} but the main focus of this
paper is on the stability and on the angles the presence of which
induces the instability of the Nystr\"om method.

Let $(A_n)$ be a bounded sequence of linear bounded operators $A_n:
S_n^d \mapsto S_n^d$. The set $\cT$ of such sequences equipped with
componentwise operations of addition, multiplication, involution and
multiplication by scalars, and with the norm
$$ \|(A_n)\| :=\sup\limits_{n\in \sN} \|A_n\|$$
becomes a $C^*$-algebra.
\begin{definition} The sequence $(A_n)\in \cT$ is called
stable if there is an $n_0\in\sN$ such that for all $n\geq n_0$ the
operators  $A_nP_n :S_n^d(\Gamma) \mapsto S_n^d(\Gamma)$ are
invertible and the norms $\|(A_nP_n)^{-1} P_n\|_{n\geq n_0}$ are
uniformly bounded.
\end{definition}

 \begin{remark}
The stability of the method is directly connected to the condition
numbers of the corresponding approximation methods. The graphs in
Figure \ref{fig:ang} show that the Nystr\"om method for the double
layer potential operator considered on the contours $\cL_1(\omega)$,
$\omega=0.25183\pi, \pi/3,\pi/4,\pi/2$ is stable. An abnormality of
the graph in the case $\omega=0.25183\pi$ is caused by the proximity
of this point to the so-called "critical angle". We refer the reader
to Section \ref{s4} for a more detailed discussion of this
phenomena.
 \end{remark}

\begin{figure}[!tb]
\includegraphics[height=45mm,width=120mm]{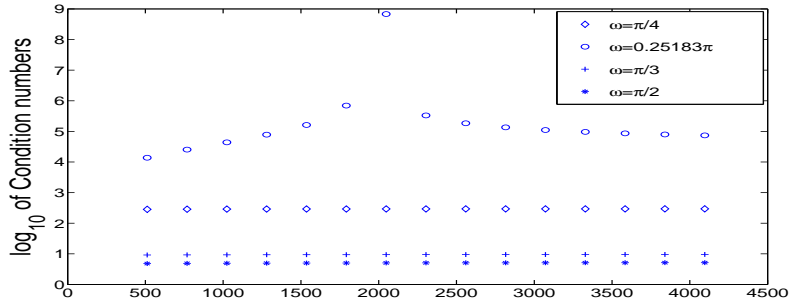}
   \caption{\sf Condition numbers for some opening angles.
   The numbers of discretization points is $16n$.}
   \label{fig:ang}
\end{figure}%
It is well known (see, for example, \cite{RSS:2011, DS:2008,
Prossdorf1991}) that the stability of the approximation method
$(A_n)$ is equivalent to the invertibility of the coset $(A_n)+ \cG$
in the quotient algebra $\cT / \cG$ where $\cG$ is the set of all
bounded sequences uniformly convergent to zero,
$$\cG = \{ (G_n)\in \cT: \liml_{n\to \infty} \|G_n\| =0\}.$$
It turns out that in many cases the quotient algebra $\cT /\cG$ is
too large to treat the invertibility problem efficiently. Therefore,
one often considers a smaller algebra $\cA\subset \cT$ of sequences
containing the approximation method in the question, at the same
time expanding the ideal $\cG$ to an ideal $\cJ$ in such a way that
the initial problem will be equivalent to the invertibility of the
corresponding coset in the quotient algebra $\cA/\cJ$. More
precisely, let $\cA \subset \cT$ denote the close subalgebra of
$\cT$ containing all sequences $(A_n)$ such the  strong limits
  $$
s-\liml_{n\to \infty}A_nP_n =A \text{  and  } s- \liml_{n\to
\infty}(A_n)^*P_n = A^*
 $$
exist. Moreover, if $\cK (L^2(\Gamma))$ is the set of all compact
operators on $L^2(\Gamma)$, then the family of the sequences
 $$
\cJ = \{ (J_n): J_n=P_nKP_n+G_n,\; K\in \cK(L^2(\Gamma)),\, G_n \in
\cG\}.
$$
is a closed two-sided ideal of $\cA$.

\sloppy

\begin{theorem}[{see \cite{RSS:2011, DS:2008, Prossdorf1991}}] \label{t1}
Assume that $(A_n)\in \cA$. Then the sequence $(A_n)$ is stable if
and only if the operator $A$ is invertible and the coset $(A_n)+\cJ$
is invertible in the quotient algebra $\cA /\cJ$.
\end{theorem}
This result can be used to study the applicability of the Nystr\"om
method to the double layer potential equations. Thus it follows from
\eqref{eqn11} that the sequence of approximation operators
$(A_n)_{n\in\sN}$ corresponding to the Nystr\"om method converges
strongly to the operator $A_\Gamma$. Similar statement is valid for
the sequence of adjoint operators. Let us show the invertibility of
the coset $(A_n)+\cJ$ in the quotient algebra $\cT/\cJ$. It can be
done by using local principles. Thus with each point $\tau\in
\Gamma$ of the contour $\Gamma$ one can associate a simpler sequence
of approximation operators $(A_n^\tau)$, and the invertibility of
the coset $(A_n)+\cJ$ in $\cT/\cJ$ is equivalent to the
invertibility of cosets containing $(A_n^\tau)$ in some algebra
associated with the point $\tau$. For more detail we refer the
reader to \cite{RSS:2011, DS:2008}. Note that if $\tau\notin
\cM_\Gamma$ and $U_\tau\subset \Gamma$ is a neighbourhood of $\tau$
such that $\cM_\Gamma\cap U_\tau=\emptyset$ and if $f_\tau$ is a
function continuous on $\Gamma$ and such that
 $$
f_\tau(t)=\left \{
\begin{array}{cl}
  1 & \quad \text{if } t=\tau \\
  0 & \quad \text{if } t\in \Gamma/U_\tau ,\\
\end{array}
 \right .
 $$
then the operator $f_\tau V_\Gamma f_\tau \in \cK (L^2(\Gamma))$
\cite[Corollary 4.6.3]{RSS:2011}). Therefore, the sequence
$(A_n^\tau)$ is locally equivalent to the sequence generated by the
projections $(P_n)$, so that the corresponding coset containing the
sequence $(A_n^\tau)$ is invertible. Thus one only has to identify
and study the cosets associated with corner points of $\Gamma$. To
this end, for each corner points $\tau_j\in\Gamma$ we consider the
corresponding approximation method for the operator $A_{\omega_j}$
of \eqref{eqn7} and approximate the integral $\int_{\Gamma_j}x(\tau)
d\tau$ by a quadrature rule similar to \eqref{eqn5}, viz.,
\begin{equation}\label{eqn14}
\begin{aligned}
 \int_{\Gamma_j}x(\tau )d\tau \approx & \suml_{l=-\infty}^{-1}\suml_{p=0}^{d-1} w_p x\left (
 \frac{l+\ve_p}{n} e^{i(\beta_j +\omega_j)}\right )\frac{e^{i(\beta_j
 +\omega_j)}}{n}\\
 & + \suml_{l=0}^{+\infty}\suml_{p=0}^{d-1}w_p x\left (
 \frac{l+\ve_p}{n} e^{i\beta_j}\right )\frac{e^{i\beta_j}}{n}
 \end{aligned}
\end{equation}
where $w_p$ and $\ve_p$ as in \eqref{eqn5}. We also need spline
spaces on the contours $\Gamma_j$ and $\sR ^+$. Let $S_n^{\beta_j,
\omega_j}$ be the smallest subspace of $L^2(\Gamma _j)$ which
contains all functions
\begin{equation*}
\widetilde {\vp}_{kn}(t) =
\begin{cases}
\begin{cases} \psi_{kn} (s) &\text{  if  } t=se^{i\beta_j} \\
0 &\text{  otherwise  }
\end{cases} , \quad k \leq 0, \vspace{0.5truecm}\\
  \begin{cases} \psi_{k-d,n}(s) &\text{  if  } t=se^{i(\beta_j + \omega
_j)}\\0 &\text{  otherwise  },
\end{cases} \quad k<0,
\end{cases}
\end{equation*}
where the basis splines $\psi_{kn}$ are defined by
 $$
\psi_{kn}(s):=\psi(ns-k), \quad s\in\sR,
 $$
and where the function $\psi:=u^d$ is obtained by recurrent
relations
 \begin{align*}
u^l(s)&=\int_\sR \chi_{[0,1)}(s-x) u^{l-1}(x)\,dx,\quad l=1,2,\ldots, d, \\
 u^0(x)&=\chi_{[0,1)}(x)= \left \{
 \begin{array}{cl}
   1 & \quad\text{if}\quad x \in [0,1)\\
   0 & \quad\text{otherwise} .\\
 \end{array}
  \right .
 \end{align*}
The spline space $S_n(\sR ^+)=S_n^d(\sR ^+)$ is constructed
similarly but we let $\beta_j =0$ and only take $\widetilde
{\vp}_{kn}$ for $k \geq 0$. Moreover, let $\widetilde{P}_n$ and
$\widehat {P}_n$ denote the orthogonal projections from
$L^2(\Gamma_j)$ onto $S_n^{\beta_j, \omega_j}$ and from $L^2(\sR
^+)$ onto $S_n(\sR ^+)$, respectively. Let $\bR_2(\Gamma_j)$ denote
the set of functions on $\Gamma_j$ which are Riemann integrable on
each finite part of $\Gamma_j$ and satisfy the condition
\begin{equation*}
\begin{aligned}
\|f\|_{\bR} &= \|f\|_{L^2(\Gamma_j)} + \left ( \suml
_{k=0}^{+\infty} \sup \limits_{t\in e^{i(\beta_j +
\omega_j)}[k,k+1]} |f(t)|^2\right )^{1/2}\\
&+\left ( \suml_{k=0}^{+\infty} \sup \limits_{t\in e^{i\beta_j
}[k,k+1]} |f(t)|^2\right )^{1/2} <+\infty.
\end{aligned}
\end{equation*}
Consider the integral equation
\begin{equation*}
A_{\Gamma_j}x=f,\quad f \in \bR_2(\Gamma_j).
\end{equation*}
As before, replace $x$ by an element $x_n \in
S_n^{\beta_j,\omega_j}$, apply quadrature formula \eqref{eqn14} to
the corresponding integrals and use the interpolation projections
$\widetilde{Q}_n^{\delta}: \bR (\Gamma_j) \mapsto S_n^{\beta_j,
\omega_j}$ defined by
\begin{equation*}
\begin{aligned}
&\widetilde {Q}_n^\delta x(t_{lp})=x(t_{lp}),\quad l\in \sZ, \;
p=0,1,\ldots , d-1; \\
&t_{lp} =
\begin{cases}
  \D \frac{l+\delta_p}{n}e^{i(\beta_j + \omega_j)} & \text{  if
 }l<0,\\[1.5ex]
 \D\frac{l+\delta_p}{n}e^{i\beta_j} & \text{  if  } l \geq 0.
\end{cases}
\end{aligned}
\end{equation*}
As the result, we obtain the following operator equations
\begin{equation}\label{eqn15}
\widetilde {Q}_n^\delta A_{\Gamma_j}^{(\ve,n)} \widetilde {P}_n x_n
= \widetilde {Q}_n^\delta f, \quad x_n \in S_n^{\beta_j,\omega
_j},\; n\in \sN.
\end{equation}
These equations are equivalent to the infinite systems of linear
algebraic equations
\begin{equation*}
\begin{aligned}
x_n(t_{kr}) &+ \frac{1}{2\pi i}\suml_{l=-\infty}^{-1}\suml
_{p=0}^{d-1} w_p x_n(t_{lp}) \left ( \frac{\tau '_{lp}}{\tau
_{lp}-t_{kr}} - \frac{\overline {\tau '_{lp}}}{\overline {\tau
_{lp}}-\overline {t_{kr}}}\right )\frac{e^{i(\beta_j + \omega
_j)}}{n}\\
&+\frac{1}{2\pi i}\suml_{l=0}^{\infty}\suml_{p=0}^{d-1} w_p
x_n(t_{lp}) \left ( \frac{\tau '_{lp}}{\tau_{lp}-t_{kr}} -
\frac{\overline {\tau '_{lp}}}{\overline {\tau_{lp}}-\overline
{t_{kr}}}\right )\frac{e^{i\beta_j}}{n}\\
&= f(t_{kr}), \quad k \in \sZ,\; p=0,1,\ldots , d-1
\end{aligned}
\end{equation*}
where $\tau_{lp}$ are defined analogously to $t_{lp}$ but the
parameter $\delta_p$ is replaced by $\ve_p$ and $\tau '_{lp} =
\gamma '(l+\ve_p)/n.$

If one  now uses the integral representation \eqref{eqn6} of the
Mellin convolution operator $\cM(\mathbf{n}_{\omega_j})$ with the
symbol $\mathbf{n}_{\omega_j}$ defined by \eqref{eqn8}, one can
write the operator \eqref{eqn15} in a different form. More
precisely, let $\widehat {Q}_n^\delta, n\in \sN$ be the
interpolation operators defined on the positive semi-axis.
 \begin{lemma}\label{l1}
If ${\bf k}_{\omega}$ is the function defined in \eqref{eqn9}, then
the sequence $(\widetilde {Q}_n^\delta A_{\Gamma_j}^{(\ve,n)}
\widetilde {P}_n)_{n\in \sN}$ is stable if and only if the sequence
$(\widehat{A}_{\omega_j}^{\ve ,\delta, n}
\diag(\widehat{P}_n,\widehat{P}_n))$,
$$ \widehat{A}_{\omega_j}^{\ve ,\delta, n} \diag(\widehat{P}_n,\widehat{P}_n) =
\begin{pmatrix}
\widehat {P}_n & \widehat {Q}_n^\delta \cM^{(\ve,n)}({\bf k}_{\omega_j})\widehat{P}_n \\
\widehat {Q}_n^{1-\delta} \cM^{(1-\ve,n)}({\bf k}_{\omega
_j})\widehat{P}_n & \widehat{P}_n
\end{pmatrix}$$
is so.
\end{lemma}
\begin{proof}
Let $\eta:L^2(\Gamma_j) \mapsto L^2(\sR ^+)^2$ be the isomorphism
defined in Section \ref{s2}. It is easily seen that
$$ \eta \widetilde{Q}_n^\delta \eta^{-1}=\diag (\widehat {Q}_n^\delta,\widehat {Q}_n^{1-\delta}),\; \eta \widetilde {P}_n\eta ^{-1}=\diag(\widehat {P}_n,\widehat {P}_n)$$
and
$$\eta A_{\Gamma_j}^{(\ve,n)}\eta ^{-1} = \eta \widetilde{P}_n\eta ^{-1} + \eta V_{\Gamma_j}^{(\ve,n)}\eta ^{-1}=
\begin{pmatrix}
\widehat{P}_n & \cM^{(\ve,n)}({\bf k}_{\omega_j}) \\
\cM^{(1-\ve, n)}({\bf k}_{\omega_j}) & \widehat{P}_n
\end{pmatrix}.
$$
The obvious identity $ \eta (\widetilde {Q}_n^\delta A_{\Gamma
_j}^{(\ve,n)} \widetilde {P}_n)\eta ^{-1} = (\eta \widetilde {Q}
_n^\delta \eta ^{-1})(\eta  A_{\Gamma_j}^{(\ve,n)}\eta ^{-1})(\eta
\widetilde {P}_n \eta ^{-1})$ completes the proof.
\end{proof}
Let $l_2$ denote the space of sequences $(\xi_j)_{j=0}^{\infty}$ of
complex numbers $\xi_j, j=0,1,\ldots$ such that $\big
(\suml_{j=0}^{\infty} |\xi_j|^2\big )^{1/2}<+\infty$. We now define
the operators $E_n: l_2\mapsto S_n(\sR ^+)$ and $E_{-n}:S_n(\sR
^+)\mapsto l_2$ by
 \begin{align*}
  E_n((\xi_j)_{j=0}^{\infty}) =\suml_{j=0}^{+\infty} \xi_j \widetilde
  {\vp}_{jn}(t),\quad
  E_{-n}\left ( \suml_{j=0}^{+\infty} \xi_j \widetilde {\vp}_{jn}(t)\right )=
(\xi_j)_{j=0}^{\infty}.
\end{align*}
Recall \cite{DeBoor1978} that the operators $E_n: l_2\mapsto S_n(\sR
^+)$ and $E_{-n}:S_n(\sR ^+)\mapsto l_2$ are bounded and there is a
constant $C$ such that
 $$
 ||E_n||||E_{-n}||\leq C \, \text{ for  all }\, n\in\sN.
  $$
The last relation allows us to write the conditions of the stability
of the sequence $(\widetilde {Q}_n^\delta A_{\Gamma_j}^{(\ve,n)}
\widetilde {P}_n)_{n\in \sN}$ in a more convenient form. By
$\widehat{E}_{n}$  and $\widehat{E}_{-n}$ we, respectively, denote
the diagonal operators,
 $$
\widehat{E}_n:=\diag (E_n, E_n), \quad \widehat{E}_{-n}:=\diag
(E_{-n}, E_{-n}).
 $$

\begin{corollary} \label{c1}
The sequence $(\widetilde {Q}_n^\delta A_{\Gamma_j}^{(\ve,n)}
\widetilde {P}_n)_{n\in \sN}$ is stable if and only if the operator
$B_{\omega_j,\delta, \ve}=\widehat{E}_{-1}\widehat{A}_{\omega
_j}^{\ve ,\delta, 1}\diag(\widehat{P}_1,
\widehat{P}_1)\widehat{E}_1$ is invertible.
\end{corollary}
\begin{proof} Straightforward calculations show that the entries
of the approximation operator $\widehat{E}_{-n}
\widehat{A}_{\omega_j}^{\ve ,\delta, n}
\diag(\widehat{P}_n,\widehat{P}_n) \widehat{E}_n$ do not depend on
$n$. Indeed, consider for example, the sequence $(E_{-n}\widehat
{Q}_n^\delta \cM^{(\ve,n)}({\bf k}_{\omega_j})\widehat{P}_n E_n)$.
If $x_n\in \im \widehat{P}_n$, then
\begin{equation*}
\begin{aligned}
(\cM^{(\ve,n)}({\bf k}_{\omega_j})x_n)(\sigma)&
=\suml_{l=0}^{+\infty} \suml_{p=0}^{d-1}w_p{\bf k}_{\omega_j}\left (
\frac{\sigma}{\frac{l+\ve_p}{n}}\right )
\frac{1}{\frac{l+\ve_p}{n}}\frac{1}{n}\, x_n\left ( \frac{l+\ve
_p}{n}\right)\\
&=\suml_{p=0}^{d-1}w_p\suml_{l=0}^{\infty}{\bf k}_{\omega_j}\left (
\frac{\sigma}{\frac{l+\ve_p}{n}}\right ) \frac{1}{l+\ve_p}\,  x_n
\left
(\frac{l+\ve_p}{n} \right ),\\
\end{aligned}
\end{equation*}
and application of the interpolation operators
$\widetilde{Q}_n^\delta$ leads to the relation
\begin{equation*}
\begin{aligned}
(\widetilde{Q}_n^\delta \cM^{(\ve,n)}({\bf k}_{\omega_j})x_n )\left
(\frac{k+\delta_r}{n}\right ) &=\suml_{p=0}^{d-1}w_p\suml
_{l=0}^{\infty}{\bf k}_{\omega_j}\left ( \frac{\frac{k+\delta
_r}{n}}{\frac{l+\ve_p}{n}}\right ) \frac{1}{l+\ve_p}\, x_n \left
(\frac{l+\ve_p}{n}
\right )\\
&=\suml_{p=0}^{d-1}w_p\suml_{l=0}^{\infty}{\bf k}_{\omega_j}\left (
\frac{k+\delta_r}{l+\ve_p}\right )  \frac{1}{l+\ve_p}\, x_n \left
(\frac{l+\ve_p}{n} \right ).
\end{aligned}
\end{equation*}
Thus the entries of the operator $\widehat{E}_{-n}
\widehat{A}_{\omega_j}^{\ve ,\delta, n}
\diag(\widehat{P}_n,\widehat{P}_n) \widehat{E}_n$ do not depend on
$n$. Therefore, the sequence in question is constant and one
concludes that it is stable if and only if one of its members, say
$E_{-1} \widehat{A}_{\omega_j}^{\ve ,\delta, 1}\diag(\widehat{P}_1,
\widehat{P}_1) E_1$, is invertible. This completes the proof.
\end{proof}

\begin{theorem} \label{stability}
Let $n=qm, m\in \sN$. Suppose that the operator $A$ is invertible.
The Nystr\"om method for the operator $A: L^2(\Gamma) \mapsto
L^2(\Gamma)$ is stable if and only if all the operators
$B_{\omega_j,\delta, \ve}, j=0,1,\ldots , d-1$ are invertible.
\end{theorem}
\begin{proof}
Let $\cC$ denote the smallest closed $C^*$-algebra that contains the
sequences $(P_n S_\Gamma P_n)$, $(P_n M P_n)$ and  $(P_n f P_n)$
where $f\in C(\Gamma)$ and let $\cJ$ be the ideal  defined in
Theorem \ref{t1}. Then $(A^{(\ve,n)}P_n)\in \cC$ and $\cC /\cJ$ is a
$C^*$-subalgebra of $\cA/\cJ$. Therefore, the coset
$(A^{(\ve,n)}P_n)+\cJ$ is invertible in $\cA /\cJ$ if and only if it
is invertible in $\cC /\cJ$. However, the algebra $\cC /\cJ$ has a
nice centre and the invertibility of the coset
$(A^{(\ve,n)}P_n)+\cJ$ in $\cC /\cJ$ can be established by the
Allan's local principle \cite{Allan1968} (see also \cite[Theorem
1.9.5]{DS:2008} for real algebra version of Allan's local
principle). Thus following the proof of Theorem 3.4 of
\cite{DH:2013a} one can show that for any $\tau=\tau_j\in
\cM_\Gamma$ this coset is invertible if and only if the
corresponding operator $B_{\omega_j,\delta, \ve}$ is invertible. On
the other hand, it was already mentioned that for $\tau\notin
\cM_\Gamma$, the corresponding coset is always invertible, and
application of Theorem \ref{t1} completes the proof.
\end{proof}

\section{Numerical approach to the invertibility of local
  operators\label{s4}}

Due to Theorem \ref{stability}, the stability of the Nystr\"om
method depends on the invertibility of the operators
$B_{\omega_j,\delta, \ve}$,  $j=0,1,\ldots, q-1$. A more detailed
study of these operators shows that they belong to an algebra of
Toeplitz operators with matrix symbols. Unfortunately, at present
there is no efficient criterion to check whether such operators are
invertible or not. On the other hand,  when considering the
stability of approximation methods for Sherman--Lauricella and
Muskhelishvili equations, a numerical approach to problems has been
proposed in \cite{DH:2011, DH:2011b}. Thus one can connect the
invertibility of $B_{\omega,\delta, \ve}$ with the stability of the
method for the corresponding initial operator $A$ on model curves,
which have one or more corner points all of the same magnitude
$\omega$. As the next step, one can check the behaviour of the
condition numbers for the method under consideration and decide
which opening angles $\omega$ belong to the set of "critical"
angles, i.e to the set of the angles which cause the instability of
the method. An essential difference to the situation with the
Muskhelishvili and Sherman--Lauricella situation is that now one
does not know whether the initial operator $A$ is invertible.
Therefore the invertibility of $A$ has to be assumed from the very
beginning or verified somehow. More precisely, one can apply Theorem
\ref{stability} in a special setting and get the following result.

  \begin{theorem}\label{t4.1}
Let $\cL=\cL(\omega)$ denote any of two curves $\cL_1(\omega)$ or
$\cL_2(\omega)$, $\omega\in (0,2\pi)$ defined in Section \ref{s3}.
If the corresponding operator $A_{\cL(\omega)}$ of \eqref{eqn12} is
invertible, then the operator $B_{\omega, \delta, \ve}$ is
invertible if and only if the Nystr\"om method $(Q_n^{\delta},
 A_{\cL(\omega)}^{(\ve,n)} P_n)$ is stable.
  \end{theorem}
  \begin{figure}[!tb] \centering
\includegraphics[height=45mm,width=60mm]{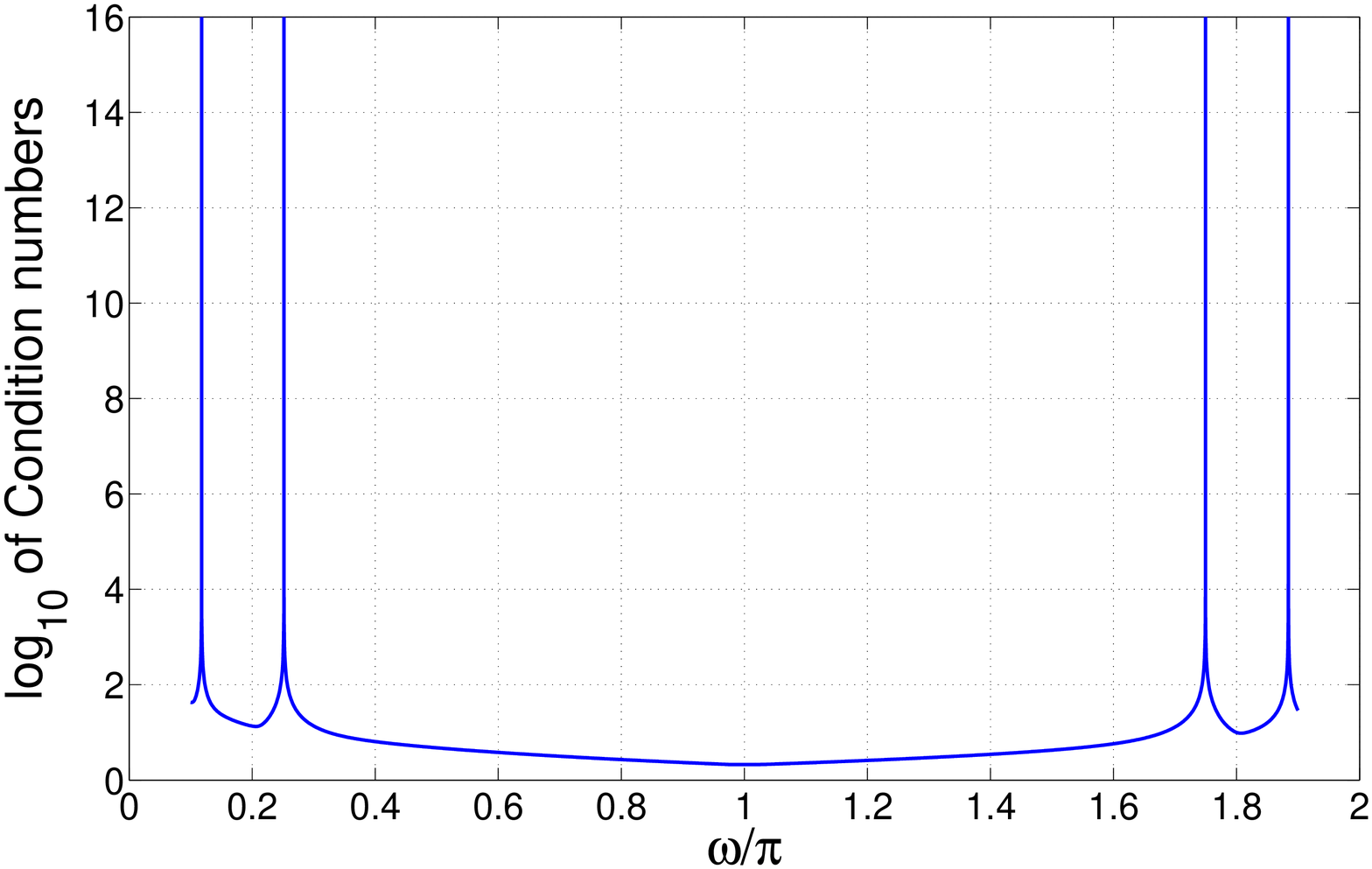}~%
 \includegraphics[height=45mm,width=60mm]{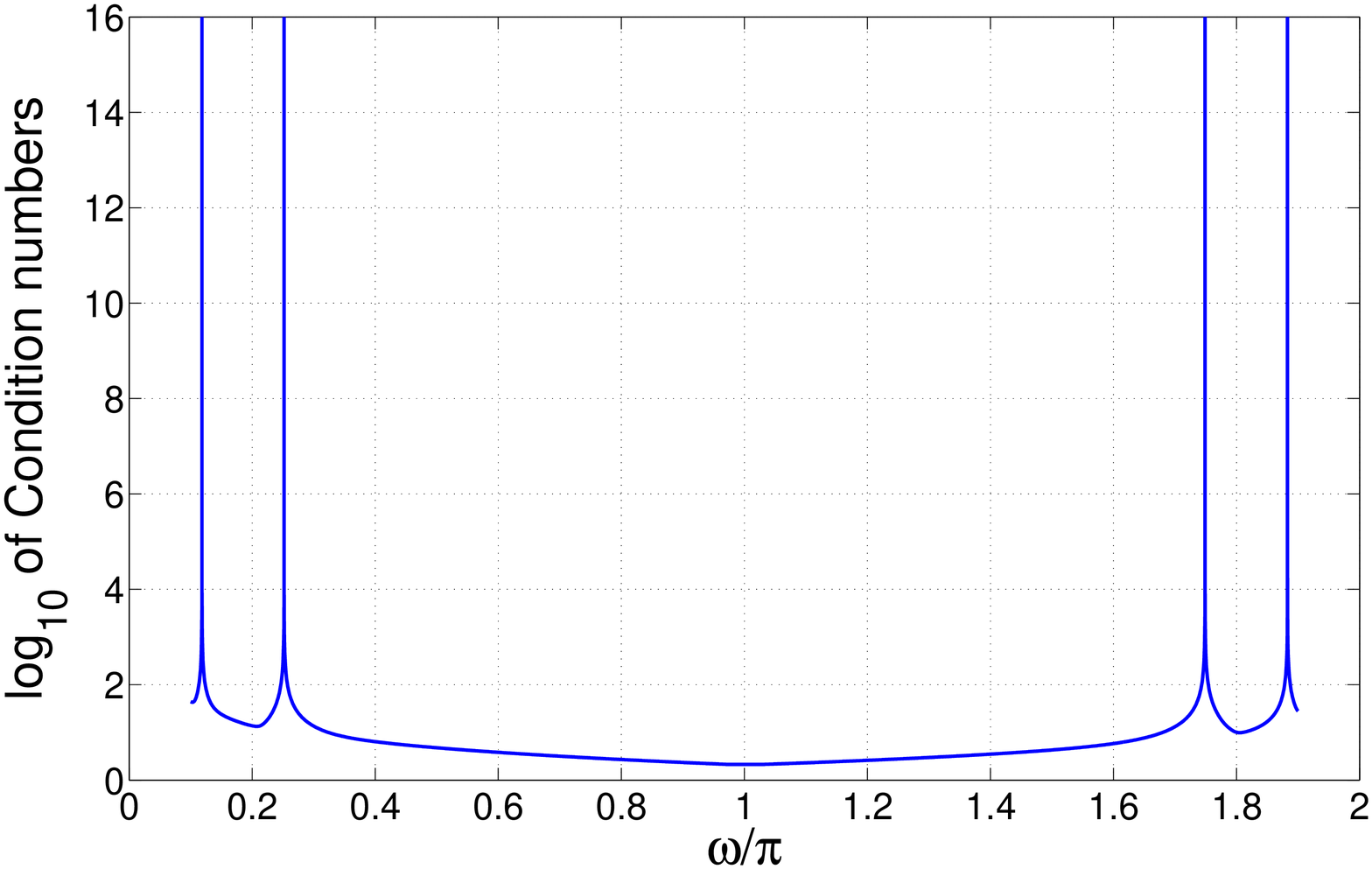}
   \caption{\sf  Condition numbers vs. opening angles
   in case $n=128,\,d=16$. Left: for one-corner curve, right: for two-corner curve}
   \label{fig:128}
\end{figure}
It is worth noting that the curves $\cL_1$ and $\cL_2$ have distinct
shapes and the number of corner points. However, if their corner
points have the same opening angle, the corresponding numerical
experiments shall produce the same results. In what follows we are
varying  parameter $\omega$ in the interval $(0.1\pi,1.9\pi)$ and
obtain two families of contours with one and two corner points,
respectively. In order to find the instability angles, we divide the
interval $[0.1\pi, 1.9\pi]$ by the points $\omega_k =
\pi*(0.1+0.001k)$. Further, for each point $\omega_k$ we compute the
condition numbers for the Nystr\"om method in the case where $n=128$
and the Gauss--Legendre quadrature with $d=16$ is used. Recall that
both sets of parameters $\ve_p$ and $\delta_p$ in \eqref{eqn13} are
the Gauss-Legendre points on the interval $[0,1]$. Calculating the
corresponding condition numbers at the points $\omega_k$, we
detected "suspicious" points in the neighbourhoods of which
condition numbers grow rapidly. Thereafter, in neighborhoods of such
points the initial mesh has been refined and condition numbers are
recalculated. The procedure is repeated until condition numbers
reach the point $10^{16}$. The outcome of these computations is
presented in Figure \ref{fig:128}. Thus using both contours we found
that the corresponding graphs have four peaks in the interval
$(0.1,1.9)$, and approximate value for the critical "angles are:

\vspace{2mm}

\centerline{\textbf{The case of one corner geometry, curve $\cL_1$}}
\begin{displaymath}
0.11781222 \pi,\quad 0.25164815 \pi,\quad 1.74949877 \pi, \quad
1.88430019 \pi
\end{displaymath}

\vspace{2mm}

\centerline{\textbf{The case of two corner geometry, curve $\cL_2$}}
\begin{displaymath}
0.11780844 \pi ,\quad 0.25164706 \pi,\quad 1.74840993 \pi, \quad
1.88390254 \pi.
\end{displaymath}

Let us emphasize that for both the curve $\cL_1$ and $\cL_2$, the
results obtained coincide up to three significant numbers. The peaks
obtained are connected with four possible critical angles in the
interval $(0.1\pi,1.9\pi)$. On the other hand, it is possible that
they arose as the result of irreversibility of the corresponding
operator $A$ on the curve $\cL_j$. However, numerical experiments
with other approximation methods for the same operators, which are
not reported here, show that those methods have distinct critical
angles. But it is not possible if the operator in question is
irreversible for the angles mentioned. Thus the values found
represent the "critical" angles of the method but not the curves
where the initial operator $A$ is not invertible.

 Note that the numerical experiments are performed in MATLAB
environment (version 7.9.0) and executed on an Acer Veriton M680
workstation equipped with a Intel Core i7 vPro 870 processor and 8GB
of RAM.

\section{Conclusion}

In this work, necessary and sufficient conditions of the stability
of the Nystr\"om method for double layer potential equations on
simple piecewise smooth contours are established. Moreover, we found
four angles in the interval $(0.1\pi, 1.9\pi)$ whose presence on the
contour $\Gamma$ will cause the instability of the method, does not
matter what the shape the curve $\Gamma$ has. Thus if the contour
$\Gamma$ possesses at least one of such angles, the Nystr\"om method
is not stable and in order to find an approximate solution of the
corresponding double layer potential equation, one has to use a
different approximation method.

The results of numerical experiments are verified by using curves
with different numbers of corner points and they are in a good
correlation with theoretical studies.


 \end{document}